\documentclass[11pt]{article}
\usepackage{amsmath,amsthm, amssymb, amsfonts,accents,nccmath}
\usepackage{enumitem}
\usepackage{array}
\usepackage[toc,page]{appendix}
\usepackage{mathrsfs}
\usepackage[english]{babel}
\usepackage{dsfont}
\usepackage{booktabs}
\usepackage[linesnumbered,commentsnumbered,ruled,vlined]{algorithm2e}
\usepackage{tikz}
\usepackage{tikz-network}
\usepackage[utf8]{inputenc}
\usepackage[english]{babel}

\usepackage{caption}
\usepackage{subcaption}
\usepackage{float}

\usepackage{hyperref}
\hypersetup{
    colorlinks=true,
    linkcolor=blue,
    filecolor=darkmagenta,      
    urlcolor=cyan,
    citecolor=red
} 
\makeatletter
\def\BState{\State\hskip-\ALG@thistlm}
\makeatother
\numberwithin{equation}{section}

\newtheorem{theorem}{Theorem}[section]

\newtheorem{lem}[theorem]{Lemma}
\newtheorem{prop}[theorem]{Proposition}

\newcommand{\bigzero}{\mbox{\normalfont\Large\bfseries 0}}

\newcommand{\ds}{\displaystyle}

\newcommand{\cof}{\textup{cof }}

\title{The exponential distance matrix of bi-block graphs}
\author{Joyentanuj Das\footnote{Department of Mathematics, College of Engineering and Technology, SRM Institute of Science and Technology, Kattankulathur, Chennai, 603203, India. \indent  Email: joyentanuj@gmail.com,  joyentad@srmist.edu.in}  \quad and \quad Sumit Mohanty\footnote{Indian Institute of Management Ranchi,  Prabandhan Nagar, Vill-Mudma, Nayasarai Road, Ranchi, Jharkhand-835303, India. \indent   Email:  sumitmath@gmail.com, sumit.mohanty@iimranchi.ac.in}}

\date{}

\begin{document}

\maketitle

\begin{abstract}
Let $G$ be a connected graph with vertex set $\{v_1, v_2, \ldots, v_\mathbf{n}\}$. As a variant of the classical distance matrix, the \emph{exponential distance matrix} was introduced independently by Yan and Yeh, and by Bapat et al. For a nonzero indeterminate $q$, the exponential distance matrix $\mathscr{F} = (\mathscr{F}_{ij})_{\mathbf{n} \times \mathbf{n}}$ of $G$ is defined by $\mathscr{F}_{ij} = q^{d_{ij}},$ where $d_{ij}$ denotes the distance between vertices $v_i$ and $v_j$ in $G$. 
A connected graph is said to be a \emph{bi-block graph} if each of its blocks is a complete bipartite graph, possibly of varying bipartition sizes. In this paper, we obtain explicit expressions for the determinant, inverse, and cofactor sum of the exponential distance matrix of bi-block graphs. As a consequence, some known results concerning the exponential distance matrix and the $q$-Laplacian matrix are generalized.
\end{abstract}



\section{Introduction  and Motivation}
Let $G=(V(G),E(G))$ denote a finite, simple, and connected graph. Here $V(G)$ is the vertex set, and $E(G)\subseteq V(G)\times V(G)$ is the edge set. When the context is unambiguous, we abbreviate the notation to $G=(V,E)$. For two vertices $i,j\in V$, we write $i\sim j$ whenever they are adjacent.  

\medskip

We now introduce some basic notation that will be used throughout this article. By ${I}_n$ we denote the $n\times n$ identity matrix, by $\mathds{1}_n$ the all-ones column vector of size $n$, and by $\mathbf{e}_i$ the standard basis vector with $1$ in the $i$th position and zeros elsewhere. The matrix ${J}_{m\times n}$ represents the $m\times n$ all-ones matrix, with ${J}_m:={J}_{m\times m}$. Likewise, $\mathbf{0}_{m\times n}$ denotes the $m\times n$ zero matrix, and simply $\mathbf{0}$ if the size is clear from context. The transpose of a matrix $A$ will be written as $A^T$.  

\medskip

A graph $G$ is called \emph{bipartite} if its vertices can be split into two disjoint sets $X$ and $Y$ such that every edge joins a vertex of $X$ with a vertex of $Y$. The complete bipartite graph $K_{s,t}$ arises when $|X|=s$, $|Y|=t$, and every vertex of $X$ is adjacent to all vertices of $Y$.  

\medskip

Endowing a connected graph $G$ with the shortest–path metric, we obtain a metric space with distance function $d(i,j)$ equal to the length of the shortest path between $i$ and $j$, and $d(i,i)=0$. Two well-studied matrices associated with this metric are the \emph{distance matrix} and the \emph{Laplacian matrix}.  

For a graph $G$ of order $n$, the distance matrix is $D(G)=[d_{ij}]$ with $d_{ij}=d(i,j)$. The Laplacian matrix of $G$, written $L(G)=[\ell_{ij}]$, is given by
\[
\ell_{ij}=
\begin{cases}
	\delta_i, & \text{if } i=j, \\
	-1, & \text{if } i\neq j \text{ and } i\sim j, \\
	0, & \text{otherwise},
\end{cases}
\]
where $\delta_i$ is the degree of vertex $i$. It is well known that $L(G)$ is symmetric and positive semidefinite. Moreover, the all-ones vector $\mathds{1}$ is an eigenvector corresponding to the eigenvalue $0$, so that $L(G)\mathds{1}=\mathbf{0}$ and $\mathds{1}^T L(G)=\mathbf{0}$ (see~\cite{Bapat}).  

\medskip

The study of distance matrices has its origin in the seminal work of Graham and Pollack~\cite{Gr1}, who proved that for any tree $T$ on $n$ vertices,  
\[
\det D(T)=(-1)^{n-1}(n-1)2^{\,n-2}.
\]
Notably, this determinant depends only on $n$ and not on the structure of the tree. Their result triggered extensive research, with several subsequent generalizations (see, e.g., \cite{Ba-Kr-Neu,Bapat0,Gr3,Gr2}).  

Later, Graham and Lov\'asz~\cite{Gr2} obtained a striking formula for the inverse of the distance matrix of a tree $T$:  
\[
D(T)^{-1} = -\tfrac{1}{2}L(T)+\frac{1}{2(n-1)}\tau\tau^T,
\qquad \text{where } \tau=(2-\delta_1,\,2-\delta_2,\,\dots,\,2-\delta_n)^T.
\]
This identity links the distance and Laplacian matrices and motivated further exploration of distance matrix inverses for wider classes of graphs. Numerous extensions are now known for weighted trees~\cite{Ba-Kr-Neu}, bidirected trees~\cite{Ba-Lal-Pati}, block graphs~\cite{Bp3}, completely positive graphs~\cite{JD}, multi-block graphs~\cite{JD1}, weighted cactoid-type digraphs~\cite{JD2}, cactoid digraphs~\cite{Hou1}, cycle–clique graphs~\cite{Hou2}, bi-block graphs~\cite{Hou3}, distance well-defined graphs~\cite{Zhou1}, and weighted cactoid digraphs~\cite{Zhou2}.  

In 2006, Yan and Yeh \cite{Yan} and Bapat et al. \cite{Bapat0} independently proposed the concept of exponential distance matrix, a variant of distance matrix. The exponential distance matrix of an \( \mathbf{n} \)-vertex graph \( G \), denoted by \( \mathscr{F} = (\mathscr{F}_{ij}) \), is the \( \mathbf{n} \times \mathbf{n} \) matrix with
\[
\mathscr{F}_{ij} = q^{d_{ij}},
\]
where \( q \) is a nonzero indeterminate. In particular, the diagonal entries of \( \mathscr{F} \) are all equal to 1. We exclude the case \( q = 0 \), because at this time the diagonal entries are undefined, while the other entries would be equal to 0. It is noted that the spectral properties of the so-called closeness matrix were recently investigated in \cite{Zheng}. Such matrix is almost the same as the exponential distance matrix \( \mathscr{F} \) with \( q = \frac{1}{2} \), while the only difference lies on the diagonal entries, which are all 1 in \( \mathscr{F} \), but all 0 in the closeness matrix.

Yan and Yeh \cite{Yan} computed the determinant of exponential distance matrix of unweighted trees. Bapat et al. \cite{Bapat0} obtained the inverse of exponential distance matrix of unweighted trees. Sivasubramanian \cite{Siva} presented a formula for the inverse of exponential distance matrix of $3$-hypertrees. A graph \( G \) is called a block graph, if each block of \( G \) is a clique (of possibly varying orders). Recently, in \cite{Xing}, Xing and Du extended the study on the exponential distance matrix from trees to block graphs. 

Our aim of this paper is to extend the study on the exponential distance matrix to bi-block graphs. More precisely, we obtain the determinant, inverse, and cofactor sum of exponential matrix for bi-block graphs, extending the results given in \cite{Bapat0,Siva,Xing,Yan}. A connected graph is called a bi-block graph if all of blocks are complete bipartite
graphs $K_{m_i,n_i}$, $i=1,2,\cdots,r$, with vertex partition of each block is $X_i \cup Y_i$, where $|X_i| = m_i$ and $|Y_i| = n_i$. If all the blocks are $K_{1,1}$, then the graph is a tree. The determinant and inverse of distance matrix of bi-block graphs were obtained by Hou and Sun in \cite{Hou3}.

\section{Notations and a few Preliminary Results}\label{sec:Notations}

Let $G$ be a bi-block graph on $\mathbf{n}$ vertices with blocks $K_{m_i, n_i}$ for $1\leq i \leq r$ with the vertex  partition $X_i \cup Y_i$ such that $|X_i|=m_i$ and $|Y_i|=n_i$. Then, $ \mathbf{n}=\sum_{i=1}^r (m_i+n_i) - (r-1).$ We now introduced two matrices $\mathbf{A}$ and $\mathbf{B}$  of  order $\mathbf{n} \times \mathbf{n}$,  and a column vector $\mu_G$  of order $\mathbf{n}$ indexed with vertex set of the bi-block graph $G$ as follows. Let  $\mathbf{A}=[a_{uv}]$ and $\mathbf{B}=[b_{uv}]$ , where
\begin{equation}\label{eqn:defn-A}
a_{uv} = \begin{cases}
\dfrac{1}{1- q^2(m_i-1)(n_i-1)},	& \text{ if } u\sim v \text{ and } u,v \in K_{m_i,n_i},\\
0, & \text{otherwise,}
\end{cases}
\end{equation}
and 
\begin{equation}\label{eqn:defn-B}
b_{uv} = \begin{cases}
	\dfrac{n_i-1}{1-q^2(m_i-1)(n_i-1)},	& \text{ if } u \neq v, u \not\sim v \text{ and } u,v \in X_i,\\
	\\
	\dfrac{m_i-1}{1-q^2(m_i-1)(n_i-1)},	& \text{ if }  u \neq v, u \not\sim v \text{ and } u,v \in Y_i,\\
	\\
	0, & \text{otherwise}.
\end{cases}
\end{equation}
Given a vertex $v$ in $G$, the $v^{th}$ entry of $\mu_G$ is given by
\begin{equation}\label{eqn:defn-mu}
\mu(v) = \sum_{v \in X_i} \frac{n_i-1}{1-q^2(m_i-1)(n_i-1)} + \sum_{v \in Y_i} \frac{m_i-1}{1-q^2(m_i-1)(n_i-1)} + (k-1),
\end{equation}
where $bi_{G}(v)=k$, is the block index of the vertex $v$. It is easy to notice that $\mathbf{A}$ is a weighted adjacency matrix of $G$, and a principal submatrix of  $\mathbf{B}$ corresponding to a block of $G$ is a weighted adjacency matrix of the complement graph of the respective block.

Without  loss of generality, let  $K_{m_r, n_r}$ be a leaf block and $H$ be the induced subgraph of $G$ on $\mathbf{m}$ vertices with blocks $K_{m_i, n_i}$ for $1\leq i \leq r-1.$  Therefore, $G= H \circledcirc K_{m_r, n_r}$ and $\mathbf{n}=\mathbf{m}+(m_r+n_r-1)$.  We will now use the notations $\mathscr{F}, \mathbf{A}, \mathbf{B}$, and $\mu$ for the bi-block graph $G$, and the corresponding notations $\widehat{\mathscr{F}}, \widehat{\mathbf{A}}, \widehat{\mathbf{B}}$, and $\widehat{\mu}$ for the subgraph $H$. Further, let us assume $v_{\mathbf{m}}$ be a cut vertex of $G$ such that $H$ and $K_{m_r, n_r}$  are connected via $v_{\mathbf{m}}$. 
 If $v_{\mathbf{m}}$ is in the vertex partition $X_r$ of $K_{m_r, n_r}$, and $u,v$ are two vertices other than $v_{\mathbf{m}}$ such that $u\in H$ and $v\in K_{m_r, n_r}$, then
$$
d(u,v) = \begin{cases}
	d(u,v_{\mathbf{m}}) + 2, 	& \text{ if } v \in X_r,\\
	d(u,v_{\mathbf{m}}) + 1,	& \text{ if } v \in Y_r.
\end{cases}
$$
Then, the $uv^{th}$ entry of the exponential distance matrix of $G$ is:
$$
\mathscr{F}_{uv} = \begin{cases}
	q^{d(u,v_{\mathbf{m}}) + 2} = q^2 \widehat{\mathscr{F}}_{i \mathbf{m}}, 	& \text{ if } v \in X_r,\\
	q^{d(u,v_{\mathbf{m}}) + 2} = q \widehat{\mathscr{F}}_{i \mathbf{m}}, 	& \text{ if } v \in Y_r.
\end{cases}
$$
For notational convenience in calculations, we will use  $K_{m_r, n_r}=K_{s, t}$. 
Let $\mathbf{E}_1$ and $\mathbf{E}_2$ represent the $\mathbf{m} \times (s-1)$ and  $\mathbf{m} \times t$ matrix, respectively, in which each entry of the last row is $1$, and all the other entries are  $0$. Let $\mathbf{E}_{\mathbf{m} \mathbf{m}}$ be the $\mathbf{m} \times \mathbf{m}$ matrix such that the only non-zero entry is at $(\mathbf{m}, \mathbf{m})$ position and the value is $1$. From these above settings we can write $\mathscr{F}, \mathbf{A}, \mathbf{B}$, and $\mu$ in the following form:

\begin{equation}\label{eqn:F}
\mathscr{F} = \left[
\begin{array}{c|c|c}
	\widehat{\mathscr{F}} & q^2 \widehat{\mathscr{F}} \mathbf{E}_1 & q \widehat{\mathscr{F}} \mathbf{E}_2\\
	\midrule
	q^2 \mathbf{E}_1^t \widehat{\mathscr{F}} & (1-q^2) I_{s-1}+q^2 J_{s-1}  & q J_{(s-1)\times t}\\
	\midrule
	q \mathbf{E}_2^t \widehat{\mathscr{F}} & q J_{t \times (s-1)} &  (1-q^2) I_{t}+ q^2 J_{t} 
\end{array}\right],
\end{equation}

$$
\mathbf{A} =  \left[
\begin{array}{c|c|c}
	\widehat{\mathbf{A}} & \mathbf{0}_{\mathbf{m} \times (s-1)} & \dfrac{1}{1-q^2(s-1)(t-1)} \mathbf{E}_2\\ \midrule
	 \mathbf{0}_{(s-1) \times \mathbf{m}} & \mathbf{0}_{s-1} & \dfrac{1}{1-q^2(s-1)(t-1)} J_{(s-1)\times t}\\ \midrule
	 \dfrac{1}{1-q^2(s-1)(t-1)} \mathbf{E}_2^t & \dfrac{1}{1-q^2(s-1)(t-1)} J_{t \times (s-1)} & \mathbf{0}_t
\end{array}\right],
$$
$$
\mathbf{B} = \left[
\begin{array}{c|c|c}
	\widehat{\mathbf{B}} & \dfrac{t-1}{1-q^2(s-1)(t-1)} \mathbf{E}_1  & \mathbf{0}_{\mathbf{m} \times t}\\ \midrule
	\dfrac{t-1}{1-q^2(s-1)(t-1)} \mathbf{E}_1^t  & \dfrac{-(t-1)}{1-q^2(s-1)(t-1)} ( I_{s-1}- J_{s-1} ) &  \mathbf{0}_{(s-1)\times t}\\ \midrule
	\mathbf{0}_{t \times \mathbf{m}} & \mathbf{0}_{t \times (s-1)} & \dfrac{-(s-1)}{1-q^2(s-1)(t-1)} (I_{t}-J_{t})
\end{array}\right],
$$
and
$$
\mu(v) = \begin{cases}
	\widehat{\mu}(v)	& \text{ if } v \in H \text{ and } v \neq v_{\mathbf{m}},\\
	\widehat{\mu}(v_{\mathbf{m}}) +  \dfrac{t-1}{1-q^2(s-1)(t-1)} -1	& \text{ if }  v = v_{\mathbf{m}},\\
	\dfrac{t-1}{1-q^2(s-1)(t-1)} 	& \text{ if } v \in X_r \text{ and } v \neq v_{\mathbf{m}},\\
	\dfrac{s-1}{1-q^2(s-1)(t-1)}	& \text{ if } v \in Y_r.
\end{cases}
$$

We now gives a few identities involving  $\mathbf{E}_1, \mathbf{E}_2$ and $ \widehat{\mathscr{F}}$. The proofs for these identities are follows from definition and hence omitted.
\begin{prop}\label{prop:id1}
Let  $ \widehat{\mathscr{F}}$, $\mathbf{E}_1, \mathbf{E}_2$ and $\mathbf{E}_{\mathbf{m} \mathbf{m}}$  be the matrices as above. Then, the following identities hold. 
\begin{enumerate}
\item [$(a)$] $\mathbf{E}_1^t \widehat{\mathscr{F}} \mathbf{E}_2 = J_{(s-1) \times t}$ and \ $\mathbf{E}_2^t \widehat{\mathscr{F}} \mathbf{E}_1 = J_{t \times (s-1)}.$

\item [$(b)$] $\mathbf{E}_1^t \widehat{\mathscr{F}} \mathbf{E}_1 = J_{(s-1)}$ and \    $\mathbf{E}_2^t \widehat{\mathscr{F}} \mathbf{E}_2 = J_{t }.$

\item [$(c)$] $\mathbf{E}_1  J_{(s-1)} =(s-1) \mathbf{E}_1$ and \ $\mathbf{E}_2  J_{t} =t \mathbf{E}_2$.

\item [$(d)$]  $\mathbf{E}_1  J_{(s-1) \times t} = (s-1) \mathbf{E}_2$   and \ $\mathbf{E}_2  J_{t \times (s-1)} = t \mathbf{E}_1$. 
\item [$(e)$]  $\mathbf{E}_1  \mathbf{E}_1^t= (s-1) \mathbf{E}_{\mathbf{m} \mathbf{m}}$ and \  $\mathbf{E}_2  \mathbf{E}_2^t= t \mathbf{E}_{\mathbf{m} \mathbf{m}}$.
\end{enumerate}
\end{prop}

We now recall a standard result on computing the determinant of block matrices. This result will be useful to prove subsequent results.

\begin{prop}\label{prop:blockdet}~\cite{Zhang}
 Let $M_{11}
   \mbox{ and } M_{22}$ be square matrices.  Then 
   \begin{small}
   $$\det \left[
\begin{array}{c|c}
M_{11} & M_{12}  \\
\midrule
\bigzero  & M_{22}
\end{array}
\right] = \det M_{11} \det M_{22}, \mbox{ and } \det \left[
\begin{array}{c|c}
M_{11} &  \bigzero \\
\midrule
 M_{21} & M_{22}
\end{array}
\right] = \det M_{11} \det M_{22}.$$
   \end{small}   
\end{prop}

Next, we prove a lemma that plays a important role in computing the determinant of the exponential matrix for bi-block graphs.

\begin{lem}
Let $G$ be a bi-block graph with blocks $K_{m_i,n_i}$ for $i=1,2,\ldots, r$  on $\mathbf{n}=\sum_{i=1}^r (m_i+n_i)-(r-1)$ vertices. 
Let $q$ be a non-zero indeterminate such that $q\neq \pm 1 $, and $\mathscr{F}$ be the exponential distance matrix of $G$. Under the assumption and notations introduced as earlier, if $\mathscr{F}$  as defined Eqn.~\eqref{eqn:F}, and $L$ is a lower triangular block matrix given by 
 \begin{equation}
 	L = \left[
\begin{array}{c|c|c}
	I_{\mathbf{m}} & \mathbf{0}_{\mathbf{m} \times (s-1)} & \mathbf{0}_{\mathbf{m} \times t}\\ \midrule
		-q^2 \mathbf{E}_1^t  & I_{s-1} & \mathbf{0}_{(s-1)\times t}\\ \midrule
		-q \mathbf{E}_2^t + \dfrac{q^3}{q^2(s-1)+1} J_{t \times (s-1)}\mathbf{E}_1^t & \dfrac{-q}{q^2(s-1)+1} J_{t \times (s-1)} & I_{t}
\end{array}\right],
 \end{equation}
 then the product $L\mathscr{F}$ is is an upper triangular block matrix and is given by
 	\begin{equation}\label{eqn:LF}
	L\mathscr{F} = \left[
\begin{array}{c|c|c}
	\widehat{\mathscr{F}} & \ast & \ast\\ \midrule
		\mathbf{0}_{(s-1)\times \mathbf{m}} & (1-q^2)(I_{s-1}+q^2 J_{s-1}) & \ast\\ \midrule
		\mathbf{0}_{t \times \mathbf{m}} & \mathbf{0}_{t \times (s-1)} & (1-q^2)I_{t} +\dfrac{q^2(q^2-1)(s-1)}{q^2(s-1)+1}J_t 
	\end{array}\right].
 \end{equation}
 Furthermore, $\det \mathscr{F}= \det (L\mathscr{F}).$
\end{lem}

\begin{proof}
Let  $$L\mathscr{F} = \left[
\begin{array}{c|c|c}
	M_{11} & \ast & \ast\\ \midrule
		M_{21} & M_{22}& \ast\\ \midrule
		M_{31} & M_{32} & M_{33}
	\end{array}\right].$$
From Eqns.~\eqref{eqn:F} and~\eqref{eqn:LF}, we get $	M_{11} = 	\widehat{\mathscr{F}} $  and $	M_{21} = -q^2 \mathbf{E}_1^t\widehat{\mathscr{F}} + I_{s-1}  q^2 \mathbf{E}_1^t 	\widehat{\mathscr{F}}=\mathbf{0}_{(s-1)\times \mathbf{m}}.$ Using the identity $\mathbf{E}_1^t \widehat{\mathscr{F}}\mathbf{E}_1=J_{s-1}$ from Proposition~\ref{prop:id1}, we get
	\begin{align}
		M_{22} &= (-q^2 \mathbf{E}_1^t)( q^2 \widehat{\mathscr{F}}\mathbf{E}_1 )+ q^2 J_{s-1} + (1-q^2) I_{s-1} \nonumber\\
		&= -q^4 \mathbf{E}_1^t \widehat{\mathscr{F}}\mathbf{E}_1 + q^2 J_{s-1} + (1-q^2) I_{s-1} \nonumber\\
		&= -q^4 J_{s-1} + q^2 J_{s-1} + (1-q^2) I_{s-1} \nonumber\\
		&= (1-q^2)(I_{s-1}+q^2 J_{s-1}). \nonumber
	\end{align}
Next,
	\begin{align}
		M_{31} &=  \left[-q \mathbf{E}_2^t + \dfrac{q^3}{q^2(s-1)+1} J_{t \times (s-1)}\mathbf{E}_1^t\right]\widehat{\mathscr{F}} - \dfrac{q}{q^2(s-1)+1} J_{t \times (s-1)} \times q^2 \mathbf{E}_1^t + q^2 \mathbf{E}_2^t 	\widehat{\mathscr{F}} \nonumber\\
		&= -  q^2 \mathbf{E}_2^t 	\widehat{\mathscr{F}} + \dfrac{q^3}{q^2(s-1)+1} J_{t \times (s-1)}\mathbf{E}_1^t \widehat{\mathscr{F}} - \dfrac{q^3}{q^2(s-1)+1} J_{t \times (s-1)}\mathbf{E}_1^t \widehat{\mathscr{F}} + q^2 \mathbf{E}_2^t 	\widehat{\mathscr{F}} \nonumber\\
		&= \mathbf{0}_{t \times \mathbf{m}}. \nonumber
	\end{align}
Using 	 the identities $\mathbf{E}_2^t \widehat{\mathscr{F}} \mathbf{E}_1 = J_{t \times (s-1)}$ and $\mathbf{E}_1^t \widehat{\mathscr{F}}\mathbf{E}_1=J_{s-1}$ from Proposition~\ref{prop:id1}, we get
	\begin{align}
		M_{32} &=  \left[-q \mathbf{E}_2^t + \dfrac{q^3}{q^2(s-1)+1} J_{t \times (s-1)}\mathbf{E}_1^t\right]  q^2 \widehat{\mathscr{F}}\mathbf{E}_1   \nonumber\\
		& \quad +\left[- \dfrac{q}{q^2(s-1)+1} J_{t \times (s-1)} \right]\left[q^2 J_{s-1} + (1-q^2) I_{s-1}\right] + q J_{t \times (s-1)} \nonumber\\
		&=  -q^3 \mathbf{E}_2^t \widehat{\mathscr{F}}\mathbf{E}_1  + \dfrac{q^5}{q^2(s-1)+1} J_{t \times (s-1)}\mathbf{E}_1^t \widehat{\mathscr{F}}\mathbf{E}_1   \nonumber\\
		& \quad - \dfrac{q^3 (s-1)}{q^2(s-1)+1} J_{t \times (s-1)} - \dfrac{q(1-q^2)}{q^2(s-1)+1} J_{t \times (s-1)}  + q J_{t \times (s-1)} \nonumber\\
		&=  -q^3 J_{t \times (s-1)}  + \dfrac{q^5 (s-1)}{q^2(s-1)+1} J_{t \times (s-1)}  - \dfrac{q^3 (s-1)}{q^2(s-1)+1} J_{t \times (s-1)} \nonumber\\
		& \quad  - \dfrac{q(1-q^2)}{q^2(s-1)+1} J_{t \times (s-1)}  + q J_{t \times (s-1)} \nonumber\\
		&= \left[-q^3  + \dfrac{q^5 (s-1)}{q^2(s-1)+1} - \dfrac{q^3 (s-1)}{q^2(s-1)+1} - \dfrac{q(1-q^2)}{q^2(s-1)+1}  + q \right] J_{t \times (s-1)}  \nonumber\\
		&= \mathbf{0}_{(s-1)\times \mathbf{m}}. \nonumber
	\end{align}
	Finally, using the identities $\mathbf{E}_2^t \widehat{\mathscr{F}} \mathbf{E}_2 = J_{t }$ and $\mathbf{E}_1^t \widehat{\mathscr{F}}\mathbf{E}_2=J_{ (s-1)\times t }$ from Proposition~\ref{prop:id1}, we get
	\begin{align}
		M_{33} &= \left[-q \mathbf{E}_2^t + \dfrac{q^3}{q^2(s-1)+1} J_{t \times (s-1)}\mathbf{E}_1^t\right] q \widehat{\mathscr{F}}\mathbf{E}_2   \nonumber\\
		& \quad +\left[- \dfrac{q}{q^2(s-1)+1} J_{t \times (s-1)} \right]  q J_{(s-1) \times t} + q^2 J_{t} + (1-q^2) I_{t} \nonumber\\
		&= -q^2 \mathbf{E}_2^t\widehat{\mathscr{F}}\mathbf{E}_2 + \dfrac{q^4}{q^2(s-1)+1} J_{t \times (s-1)}\mathbf{E}_1^t \widehat{\mathscr{F}}\mathbf{E}_2   \nonumber\\
		& \quad - \dfrac{q^2}{q^2(s-1)+1} J_{t \times (s-1)}  J_{(s-1) \times t} + q^2 J_{t} + (1-q^2) I_{t} \nonumber\\
		&= -q^2 J_t + \dfrac{q^4}{q^2(s-1)+1} J_{t \times (s-1)} J_{(s-1) \times t}  \nonumber\\
		& \quad - \dfrac{q^2}{q^2(s-1)+1} J_{t \times (s-1)} J_{(s-1) \times t} + q^2 J_{t} + (1-q^2) I_{t} \nonumber\\
		&= -q^2 J_t + \dfrac{q^4(s-1)}{q^2(s-1)+1} J_t  - \dfrac{q^2(s-1)}{q^2(s-1)+1} J_t + q^2 J_{t} + (1-q^2) I_{t} \nonumber\\
		&= (1-q^2)I_{t} +\dfrac{q^2(q^2-1)(s-1)}{q^2(s-1)+1}J_t . \nonumber
	\end{align}
Hence, the desired result follows.
\end{proof}

\section{Determinant}
In this section, we will compute the determinant of the exponential matrix for bi-block graphs. Before proceeding further, we first recall a few known results from matrix theory which will be used in subsequent results. Let $M$ be an $n\times n$ matrix partitioned as
\begin{equation}\label{eqn:B}
M= \left[
\begin{array}{c|c}
M_{11}& M_{12} \\
\midrule
M_{21} &M_{22}
\end{array}
\right],
\end{equation}
where $ M_{11}$ and $M_{22}$ are square matrices. If  $M_{11}$ is nonsingular, then the Schur complement of $M_{11}$ in $M$ is defined to be the matrix $M/M_{11}= M_{22}-M_{21}M_{11}^{-1}M_{12}$. We now state a result that determines the determinant and inverse of a partitioned matrix using Schur complement.
\begin{prop}\label{prop:schur_det_inv}\cite{Zhang1}
	Let $M$ be a nonsingular matrix partitioned as in Eqn~(\ref{eqn:B}). If  $M_{11}$ is square and invertible, then $$\det(M) = \det(M_{11}) \det(M/M_{11}).$$
Moreover, if $M$ is a nonsingular matrix, then	
$$M^{-1}=\left[
\begin{array}{c|c}
M_{11}^{-1}+ M_{11}^{-1}M_{12}(M/M_{11})^{-1}M_{21}M_{11}^{-1}& -M_{11}^{-1} M_{12}(M/M_{11})^{-1} \\
\midrule
-(M/M_{11})^{-1}M_{21}M_{11}^{-1} & (M/M_{11})^{-1}
\end{array}
\right].$$
\end{prop}


Next we state a result without proof that gives the  eigenvalues and inverse for matrix of the  form $aI+bJ$.
\begin{lem}\label{lem:aI+bJ}
Let  $n \geq 2$   and $J_n$, $I_n$ be matrices as defined before. For $a \neq 0$, the eigenvalues of $aI_n+bJ_n$ are $a$ with multiplicity $n-1$, $a+nb$ with multiplicity $1$ and the determinant is given by $a^{n-1}(a+nb)$. Moreover, the matrix is invertible if and only if $a+nb \neq 0$ and the inverse is given by $$(aI_n + bJ_n)^{-1} = \frac{1}{a} \left(I_n - \frac{b}{a+nb} J_n\right).$$
\end{lem}

We now compute the determinant of the exponential matrix for a complete bipartite graph.

\begin{lem}\label{lem:F-1-det}
	Let $\mathscr{F}$ be the exponential distance matrix of complete bipartite graph $K_{s,t}$. If  $q$ is a nonzero indeterminate such that $q\neq \pm 1 $, then	 
	$$\det(\mathscr{F})= (1-q^2)^{s+t-1} \left[1-q^2(s-1)(t-1)\right].$$
\end{lem}

\begin{proof}
The exponential distance matrix of $K_{s,t}$ has the following block form:
$$ \mathscr{F} = \begin{bmatrix}
		 (1-q^2) I_{s}+q^2 J_{s}  & q J_{s\times t}\\
		q J_{t \times s} &  (1-q^2) I_{t}+q^2 J_{t} 
	\end{bmatrix}= \left[
\begin{array}{c|c}
M_{11}& M_{12} \\
\midrule
M_{21} &M_{22}
\end{array}\right].	$$
Using Lemma~\ref{lem:aI+bJ}, $\det M_{11}=(1-q^2)^{s-1}[q^2(s-1)+1]\neq 0,$ and hence
\begin{equation}\label{eqn:M11_inv}
M_{11}^{-1}= \dfrac{1}{1-q^2}\left[ I_s- \dfrac{q^2}{q^2(s-1)+1}J_s\right].
\end{equation}
The Schur complement of $M_{11}$ in $\mathscr{F}$ is given by
\begin{align}\label{eqn:shur_comp}
\mathscr{F}/M_{11}&= M_{22}-M_{21}M_{11}^{-1}M_{12} \nonumber\\
&= \left[(1-q^2) I_{t}+q^2 J_{t} \right]- \dfrac{q^2}{1-q^2}\left[J_{t \times s}\left(I_s- \dfrac{q^2}{q^2(s-1)+1}J_s\right)J_{s\times t}\right] \nonumber\\
&= \left[(1-q^2) I_{t}+q^2 J_{t} \right]- \dfrac{q^2}{1-q^2}\left[s- \dfrac{q^2s^2}{q^2(s-1)+1}\right]J_t \nonumber\\
&=(1-q^2) I_{t} -  \dfrac{q^2(1-q^2)(s-1)}{q^2(s-1)+1}J_t.
\end{align}
Thus, using Lemma~\ref{lem:aI+bJ}, we get
\begin{align*}
\det (\mathscr{F}/M_{11})&= (1-q^2)^{t-1} \left[(1-q^2) - t \  \dfrac{q^2(1-q^2)(s-1)}{q^2(s-1)+1} \right]= (1-q^2)^{t}\left[ \dfrac{1-q^2(s-1)(t-1)}{q^2(s-1)+1}\right].
\end{align*}
By Proposition~\ref{prop:schur_det_inv}, $\det \mathscr{F}=\det M_{11} \det (\mathscr{F}/M_{11})$ and hence the desired result follows.
\end{proof}

\begin{theorem} 
Let $G$ be a bi-block graph with blocks $K_{m_i,n_i}$ for $i=1,2,\ldots, r$  on $\mathbf{n}=\sum_{i=1}^r (m_i+n_i)-(r-1)$ vertices. If  $q$ is a non-zero indeterminate such that $q\neq \pm 1 $, and $\mathscr{F}$ be the exponential distance matrix of $G$, then, 
$$\det(\mathscr{F})= (1-q^2)^{\mathbf{n}-1}\prod_{i=1}^{r} \left[1-q^2(m_i-1)(n_i-1)\right].$$
\end{theorem}
\begin{proof}
We will prove this result using induction on the number of blocks $r$ in $G$. From Lemma~\ref{lem:F-1-det}, the result is hold true for $r=1$. Let us assume that the result is true for bi-block graphs with $r-1$ blocks.

 Let  $K_{s,t}=K_{m_r, n_r}$ be a leaf block and $H$ be the induced subgraph of $G$ on $\mathbf{m}$ vertices with blocks $K_{m_i, n_i}$ for $1\leq i \leq r-1.$ Under the notations defined in Section~\ref{sec:Notations}, $\widehat{\mathscr{F}}$ denote the exponential distance matrix of $H$ and  the exponential distance matrix $\mathscr{F}$ of $G$ can be written as in Eqn.~\eqref{eqn:F}. Thus,  $\mathbf{n} = \mathbf{m} + (s+t-1)$. Applying  induction hypothesis on $H$, we have 
 \begin{equation}\label{eqn:det-f-hat}
	\det(\widehat{\mathscr{F}}) = (1-q^2)^{\mathbf{m}-1} \prod_{i=1}^{r-1} \left[1-q^2(m_i-1)(n_i-1)\right].
 \end{equation}
From Lemma~\ref{eqn:LF}, $\det \mathscr{F}= \det (L\mathscr{F}),$ where $L\mathscr{F}$ is the upper triangular block matrix defined in Eqn.~\eqref{eqn:LF}. Using Proposition~\ref{prop:blockdet} and Eqn.~\eqref{eqn:LF}, we have
\begin{align*}
	\det(\mathscr{F}) &= \det(L\mathscr{F})\\ 
	                  &= \det(\hat{\mathscr{F}}) \times \det\Big[ (1-q^2)(I_{s-1}+ q^2 J_{s-1})\Big] \times \det \left[(1-q^2)I_{t} + \dfrac{q^2(q^2-1)(s-1)}{q^2(s-1)+1}J_t  \right].
\end{align*}
Using Lemma~\ref{lem:aI+bJ}, we have 
\begin{align*}
& \det\Big[ (1-q^2)(I_{s-1}+ q^2 J_{s-1})\Big]= (1-q^2)^{s-1} [q^2(s-1)+1]\\
\mbox{ and } & \det \left[(1-q^2)I_{t} + \dfrac{q^2(q^2-1)(s-1)}{q^2(s-1)+1}J_t \right]= (1-q^2)^t \left[\frac{1-q^2(s-1)(t-1)}{q^2(s-1)+1} \right]
\end{align*}
Thus, using Eqn.~\eqref{eqn:det-f-hat}, we get
	\begin{equation*}
	\begin{split}
		\det(\mathscr{F}) & = \det(\widehat{\mathscr{F}}) \times (1-q^2)^{s-1}  [q^2(s-1)+1] \times (1-q^2)^t \left[\frac{1-q^2(s-1)(t-1)}{q^2(s-1)+1} \right]\\
		&= \det(\widehat{\mathscr{F}}) \times (1-q^2)^{s+t-1}  \left[1-q^2(s-1)(t-1) \right]\\
		&= \left[(1-q^2)^{\mathbf{m}-1} \prod_{i=1}^{r-1} \left[1-q^2(m_i-1)(n_i-1)\right] \right] \times (1-q^2)^{s+t-1}  \left[1-q^2(s-1)(t-1) \right]\\
		&= (1-q^2)^{\mathbf{m} + (s+t-1)}  \left[1-q^2(s-1)(t-1)\right] \prod_{i=1}^{r} \left[1-q^2(m_i-1)(n_i-1)\right].
	\end{split}
	\end{equation*}
Since $\mathbf{n} = \mathbf{m} + (s+t-1)$ and $K_{m_r,n_r} = K_{s,t}$, we get
$$\det(\mathscr{F}) = (1-q^2)^{\mathbf{n}-1} \prod_{i=1}^{r} \left[1-q^2(m_i-1)(n_i-1)\right].$$
This completes the proof.
\end{proof}

\section{Inverse}

In this section, we will find the inverse of the exponential matrix for bi-block graphs. We begin with a lemma that compute the inverse of the exponential matrix of a complete bipartite graph.

\begin{lem}\label{lem:F-1-inv}
	Let $\mathscr{F}$ be the exponential distance matrix of complete bipartite graph $K_{s,t}$. If  $q$ is a nonzero indeterminate such that $q\neq \pm 1 $ and $q^2(s-1)(t-1)\neq 1$, then	
	$$\mathscr{F}^{-1} = \frac{q}{(1-q^2) \left[1-q^2(s-1)(t-1)\right]}  \begin{bmatrix}
		q(t-1)J_s & -J_{s \times t}\\
		- J_{t \times s} & q(s-1)J_t 
	\end{bmatrix} + \frac{1}{1-q^2} I.$$ 
\end{lem}

\begin{proof}
The exponential distance matrix of $K_{s,t}$ has the following block form:
$$ \mathscr{F}= \begin{bmatrix}
		 (1-q^2) I_{s}+q^2 J_{s}  & q J_{s\times t}\\
		q J_{t \times s} &  (1-q^2) I_{t}+q^2 J_{t} 
	\end{bmatrix}= \left[
\begin{array}{c|c}
M_{11}& M_{12} \\
\midrule
M_{21} &M_{22}
\end{array}\right].	$$
From Eqn.~\eqref{eqn:M11_inv}, we have $M_{11}^{-1}= \dfrac{1}{1-q^2}\left[ I_s- \dfrac{q^2}{q^2(s-1)+1}J_s\right].$ Then, it is easy to verify that 
\begin{equation}\label{eqn:eq1}
M_{21}M_{11}^{-1}=\dfrac{q}{q^2(s-1)+1}   J_{t\times s} \mbox{ and }M_{11}^{-1}M_{12}=\dfrac{q}{q^2(s-1)+1} J_{s\times t}.
\end{equation}
From Eqn.~\eqref{eqn:shur_comp}, we have 
$\mathscr{F}/M_{11}=(1-q^2) I_{t} -  \dfrac{q^2(1-q^2)(s-1)}{q^2(s-1)+1}J_t$ and hence using Lemma~\ref{lem:aI+bJ}, we get 
\begin{equation}\label{eqn:eq2}
(\mathscr{F}/M_{11})^{-1}=\dfrac{1}{1-q^2}\left[I_t-\dfrac{q^2(s-1)}{1-q^2(s-1)(t-1)}J_t\right]
\end{equation}
Therefore, using Eqns.~\eqref{eqn:eq1} and~\eqref{eqn:eq2}, we have
\begin{align}\label{eqn:eq3}
(\mathscr{F}/M_{11})^{-1}M_{21}M_{11}^{-1}&= \dfrac{q}{(1-q^2)(q^2(s-1)+1)}\left[I_t-\dfrac{q^2(s-1)}{1-q^2(s-1)(t-1)}J_t\right]  J_{t\times s} \nonumber\\
&= \dfrac{q}{(1-q^2)(q^2(s-1)+1)}\left[1-\dfrac{q^2t(s-1)}{1-q^2(s-1)(t-1)}\right]  J_{t\times s} \nonumber\\
&= \dfrac{q}{(1-q^2)[1-q^2(s-1)(t-1)]} J_{t\times s},
\end{align}
and 
\begin{align}\label{eqn:eq4}
M_{11}^{-1}M_{12}(\mathscr{F}/M_{11})^{-1}&= \dfrac{q}{(1-q^2)(q^2(s-1)+1)}J_{s\times t}\left[I_t-\dfrac{q^2(s-1)}{1-q^2(s-1)(t-1)}J_t\right]\nonumber\\
&= \dfrac{q}{(1-q^2)(q^2(s-1)+1)}\left[1-\dfrac{q^2t(s-1)}{1-q^2(s-1)(t-1)}\right] J_{s\times t} \nonumber\\
&= \dfrac{q}{(1-q^2)[1-q^2(s-1)(t-1)]} J_{s\times t}.
\end{align}
Finally, using above calculations and Eqns.~\eqref{eqn:eq1}, we get 
\begin{align*}
M_{11}^{-1}M_{12}(\mathcal{F}/M_{11})^{-1}M_{21}M_{11}^{-1}&=\left[\dfrac{q}{(1-q^2)[1-q^2(s-1)(t-1)]} J_{s\times t}\right]\left[\dfrac{q}{q^2(s-1)+1}   J_{t\times s}\right]\\
&=\left[\dfrac{q^2t}{(1-q^2)[1-q^2(s-1)(t-1)](q^2(s-1)+1)} \right]J_s,
\end{align*}
and hence
\begin{align}\label{eqn:eq5}
& M_{11}^{-1}+M_{11}^{-1}M_{12}(\mathscr{F}/M_{11})^{-1}M_{21}M_{11}^{-1} \nonumber\\=& \dfrac{1}{1-q^2}\left[ I_s- \dfrac{q^2}{q^2(s-1)+1}J_s\right]
 +\left[\dfrac{q^2t}{(1-q^2)[1-q^2(s-1)(t-1)](q^2(s-1)+1)} \right]J_s \nonumber\\
 =& \dfrac{1}{1-q^2} I_s - \dfrac{q^2}{(1-q^2)(q^2(s-1)+1)}\left[\dfrac{t}{1-q^2(s-1)(t-1)}-1\right]J_s \nonumber\\
 =& \dfrac{1}{1-q^2} I_s - \dfrac{q^2(t-1)}{(1-q^2)[1-q^2(s-1)(t-1)]}J_s
\end{align}
In view of Eqns.~\eqref{eqn:eq2} -~\eqref{eqn:eq5}, the result follows from Proposition~\ref{prop:schur_det_inv}. 
\end{proof}

We now ready to compute the inverse of the exponential matrix of a bi-block graph whenever it exists.

\begin{theorem}\label{thm:inv}
Let $G$ be a bi-block graph with blocks $K_{m_i,n_i}$ for $i=1,2,\ldots, r$  on $\mathbf{n}=\sum_{i=1}^r (m_i+n_i)-(r-1)$ vertices. Let $q$ be a non-zero indeterminate and $\mathscr{F}$ be the exponential distance matrix of $G$. If  $\mathscr{F}$ is invertible, then
	\[
	\mathscr{F}^{-1} = \frac{1}{1-q^2} I_{\mathbf{n}} - \frac{q}{1-q^2}\mathbf{A} +\frac{q^2}{1-q^2} \mathbf{B}  +\frac{q^2}{1-q^2} \textup{diag}(\mu),
		\]
where 	$\mathbf{A}, \mathbf{B}$ and $\mu$ as defined in Eqns.~\eqref{eqn:defn-A} -~\eqref{eqn:defn-mu}.
\end{theorem}
\begin{proof}
We will prove this result using induction on the number of blocks $r$ in $G$. If $r=1$, then $G$ is a complete bipartite graph. Let $G=K_{s,t}$ with the vertex partition $X\cup Y$ such that $|X|=s$ and $|Y|=t$. Then, $\mathbf{n}=s+t$ and using  Eqns.~\eqref{eqn:defn-A} -~\eqref{eqn:defn-mu}, we get

	$$
	\mathbf{A} = \dfrac{1}{1-q^2(s-1)(t-1)}\left[
\begin{array}{c|c}
		\mathbf{0}_{s} &  J_{s\times t}\\ \midrule
		J_{t \times s} & \mathbf{0}_t
	\end{array} \right],
	$$
	$$
	\mathbf{B} = \dfrac{1}{1-q^2(s-1)(t-1)} \left[
\begin{array}{c|c}
		(t-1)(J_{s} - I_{s}) &  \mathbf{0}_{s\times t}\\ \midrule
		\mathbf{0}_{t \times s} & (s-1)(J_{t} - I_{t})
	\end{array} \right],
	$$
	and
	\[
	\mu(v) = \begin{cases}
		\dfrac{t-1}{1-q^2(s-1)(t-1)}, 	& \text{ if } v \in X,\\
		\dfrac{s-1}{1-q^2(s-1)(t-1)},	& \text{ if } v \in Y.
	\end{cases}
	\]
Therefore, using the above equations and Lemma~\ref{lem:F-1-inv}, we get
\begin{align*}
& \ \frac{1}{1-q^2} I_{\mathbf{n}} - \frac{q}{1-q^2}\mathbf{A} +\frac{q^2}{1-q^2} \mathbf{B}  +\frac{q^2}{1-q^2} \textup{diag}(\mu)\\
=& \frac{1}{1-q^2} I_{\mathbf{n}} - \dfrac{q}{(1-q^2)[1-q^2(s-1)(t-1)]}\left[
\begin{array}{c|c}
		\mathbf{0}_{s} &  J_{s\times t}\\ \midrule
		J_{t \times s} & \mathbf{0}_t
	\end{array} \right]\\
	 & \quad + \dfrac{q^2}{(1-q^2)[1-q^2(s-1)(t-1)]} \left[
\begin{array}{c|c}
		(t-1)(J_{s} - I_{s}) &  \mathbf{0}_{s\times t}\\ \midrule
		\mathbf{0}_{t \times s} & (s-1)(J_{t} - I_{t})
	\end{array} \right]\\
	& \quad  + \dfrac{q^2}{(1-q^2)[1-q^2(s-1)(t-1)]}  \left[ \begin{array}{c|c}
		(t-1) I_{s} &  \mathbf{0}_{s\times t}\\ \midrule
		\mathbf{0}_{t \times s} & (s-1)I_{t}
	\end{array} \right]\\
	=&  \frac{1}{1-q^2} I_{\mathbf{n}}+ \frac{q}{(1-q^2) \left[1-q^2(s-1)(t-1)\right]}  \left[
\begin{array}{c|c}
q(t-1)J_s & -J_{s \times t}\\ \midrule
		- J_{t \times s} & q(s-1)J_t 
		\end{array} \right]
		= \mathscr{F}^{-1}
\end{align*}
This proves the result hold true for $r=1.$  Let us assume that the result is true for bi-block graphs with $r-1$ blocks.

Let $G$ be a bi-block graph with blocks $K_{m_i,n_i}$ for $i=1,2,\ldots, r$  on $\mathbf{n}=\sum_{i=1}^r (m_i+n_i)-(r-1)$ vertices. Without loss generality, let  $K_{s,t}=K_{m_r, n_r}$ be a leaf block and $H$ be the induced subgraph of $G$ on $\mathbf{m}$ vertices with blocks $K_{m_i, n_i}$ for $1\leq i \leq r-1.$ Recall that, the notations defined in Section~\ref{sec:Notations}, we will use   the notations $\mathscr{F}, \mathbf{A}, \mathbf{B}$, and $\mu$ for the bi-block graph $G$, and the corresponding notations $\widehat{\mathscr{F}}, \widehat{\mathbf{A}}, \widehat{\mathbf{B}}$, and $\widehat{\mu}$ for the subgraph $H$. Then, the exponential distance matrix $\mathscr{F}$ of $G$ can be written as in Eqn.~\eqref{eqn:F} and  $\mathbf{n} = \mathbf{m} + (s+t-1)$. Applying  induction hypothesis on $H$, we have 
\begin{equation}
\widehat{\mathscr{F}}^{-1} = \frac{1}{1-q^2} I_{\mathbf{m}} - \frac{q}{1-q^2}\widehat{\mathbf{A}} +\frac{q^2}{1-q^2} \widehat{\mathbf{B}}  +\frac{q^2}{1-q^2} \textup{diag}(\widehat{\mu)}.
\end{equation}
Let us partition the $\mathscr{F}^{-1}$ into block matrix conformal to the block matrix form of $\mathscr{F}$ given in~Eqn.~\eqref{eqn:F}, \textit{i.e.}
	\[
	\mathscr{F}^{-1} = \begin{bmatrix}
		A_{11} & A_{12} & A_{13}\\
		A_{21} & A_{22} & A_{23}\\
		A_{31} & A_{32} & A_{33}
	\end{bmatrix},
	\]
where $A_{21} = A_{12}^t, A_{31} = A_{13}^t$ and $A_{32} = A_{23}^t$. That is,
\begin{equation}\label{eqn:prod}
\left[
\begin{array}{c|c|c}
	\widehat{\mathscr{F}} & q^2 \widehat{\mathscr{F}} \mathbf{E}_1 & q \widehat{\mathscr{F}} \mathbf{E}_2\\
	\midrule
	q^2 \mathbf{E}_1^t \widehat{\mathscr{F}} & (1-q^2) I_{s-1}+q^2 J_{s-1}  & q J_{(s-1)\times t}\\
	\midrule
	q \mathbf{E}_2^t \widehat{\mathscr{F}} & q J_{t \times (s-1)} &  (1-q^2) I_{t}+ q^2 J_{t} 
\end{array}\right]\begin{bmatrix}
	A_{11} & A_{12} & A_{13}\\
	A_{21} & A_{22} & A_{23}\\
	A_{31} & A_{32} & A_{33}
	\end{bmatrix} =  \begin{bmatrix}
	I_{\mathbf{m}} & \mathbf{0} & \mathbf{0}\\
	\mathbf{0} & I_{s-1} & \mathbf{0}\\
	\mathbf{0} & \mathbf{0} & I_{t}
	\end{bmatrix}.
\end{equation}
Considering the multiplication of first row $\mathscr{F}$ with $\mathscr{F}^{-1}$, we get	
	\begin{align*}
	& \widehat{\mathscr{F}}A_{11} + q^2 \hat{\mathscr{F}} \mathbf{E}_1 A_{21} + q \widehat{\mathscr{F}}  \mathbf{E}_2 A_{31} = I_{\mathbf{m}},\\	
	& 	\widehat{\mathscr{F}} A_{12} + q^2 \widehat{\mathscr{F}}  \mathbf{E}_1 A_{22} + q \widehat{\mathscr{F}}  \mathbf{E}_2 A_{32} = \mathbf{0}_{\mathbf{m} \times (s-1)},\\	
	& 	\widehat{\mathscr{F}} A_{13} + q^2 \widehat{\mathscr{F}}  \mathbf{E}_1 A_{23} + q \widehat{\mathscr{F}}  \mathbf{E}_2 A_{33} = \mathbf{0}_{\mathbf{m} \times t}.
		\end{align*}
Since $\widehat{\mathscr{F}}^{-1}$ is invertible, the above equations reduces to
	\begin{align}\label{eqn:A1j}
	\begin{cases}
	\vspace*{.1cm}
	& A_{11} = \widehat{\mathscr{F}}^{-1} - q^2  \mathbf{E}_1 A_{21} - q   \mathbf{E}_2 A_{31} ,\\	
	\vspace*{.1cm}
	& 	 A_{12}  = - q^2 \mathbf{E}_1 A_{22} - q   \mathbf{E}_2 A_{32},\\	
		& 	A_{13}  = - q^2  \mathbf{E}_1 A_{23} - q   \mathbf{E}_2 A_{33}. 
	\end{cases}
	\end{align}	
In Eqn.~\eqref{eqn:prod}, considering the second row of $\mathscr{F}$	and the second column of $\mathscr{F}^{-1}$, we get
$$I_{s-1}=q^2 \mathbf{E}_1^t \widehat{\mathscr{F}} A_{12}+ [(1-q^2) I_{s-1}+q^2 J_{s-1}] A_{22} +  q J_{(s-1)\times t} A_{32}. $$
Substituting $A_{12}$ from Eqn.~\eqref{eqn:A1j}, and using the identities  $\mathbf{E}_1^t \widehat{\mathscr{F}}\mathbf{E}_1=J_{s-1}$ and $ \mathbf{E}_1^t  \widehat{\mathscr{F}} \mathbf{E}_2=J_{(s-1)\times t}$ from Proposition~\ref{prop:id1}, we have
\begin{align*}
I_{s-1} &= q^2 \mathbf{E}_1^t \widehat{\mathscr{F}} [- q^2 \mathbf{E}_1 A_{22} - q   \mathbf{E}_2 A_{32}]+ [(1-q^2) I_{s-1}+q^2 J_{s-1}] A_{22} +  q J_{(s-1)\times t} A_{32}\\
&= -q^4 \mathbf{E}_1^t \widehat{\mathscr{F}}\mathbf{E}_1 A_{22} - q^3 \mathbf{E}_1^t  \widehat{\mathscr{F}} \mathbf{E}_2 A_{32} + [(1-q^2) I_{s-1}+q^2 J_{s-1}] A_{22} +  q J_{(s-1)\times t} A_{32}\\
&= -q^4 J_{s-1} A_{22} - q^3 J_{(s-1)\times t} A_{32} + [(1-q^2) I_{s-1}+q^2 J_{s-1}] A_{22} +  q J_{(s-1)\times t} A_{32},
\end{align*} 
and hence $(1-q^2)[I_{s-1}+q^2J_{s-1}] A_{22} = I_{s-1}-q(1-q^2)J_{(s-1)\times t} A_{32}$. Uisng  Lemma~\ref{lem:aI+bJ}, we have $[I_{s-1}+q^2J_{s-1}]^{-1}= I_{s-1} -\dfrac{q^2}{q^2(s-1)+1}J_{s-1}$ which implies that
\begin{align}\label{eqn:A22-1}
A_{22} &= \dfrac{1}{1-q^2}\left[I_{s-1} -\dfrac{q^2}{q^2(s-1)+1}J_{s-1}\right]\left[I_{s-1}-q(1-q^2)J_{(s-1)\times t} A_{32}\right] \nonumber\\ 
&= \dfrac{1}{1-q^2}\left[I_{s-1}- \dfrac{q^2}{q^2(s-1)+1}J_{s-1} -\dfrac{q(1-q^2)}{q^2(s-1)+1} J_{(s-1)\times t} A_{32} \right].
\end{align}
In Eqn.~\eqref{eqn:prod}, considering the third column of $\mathscr{F}$ and the second column of $\mathscr{F}^{-1}$, we get
$$ \mathbf{0}_{t \times (s-1)}= q \mathbf{E}_2^t \widehat{\mathscr{F}} A_{12} + q J_{t \times (s-1)} A_{22} + [(1-q^2) I_{t}+ q^2 J_{t}] A_{32}.$$
Substituting $A_{12}$ from Eqn.~\eqref{eqn:A1j}, and using  the identities $\mathbf{E}_2^t \widehat{\mathscr{F}}\mathbf{E}_1=J_{t\times (s-1)}$ and $ \mathbf{E}_2^t  \widehat{\mathscr{F}} \mathbf{E}_2=J_{t}$ from Proposition~\ref{prop:id1}, we have
\begin{align*}
\mathbf{0}_{t \times (s-1)}&= q \mathbf{E}_2^t \widehat{\mathscr{F}} [- q^2 \mathbf{E}_1 A_{22} - q   \mathbf{E}_2 A_{32}] + q J_{t \times (s-1)} A_{22} + [(1-q^2) I_{t}+ q^2 J_{t}] A_{32}\\
&= -  q^3 \mathbf{E}_2^t \widehat{\mathscr{F}}  \mathbf{E}_1 A_{22} -  q^2 \mathbf{E}_2^t \widehat{\mathscr{F}} \mathbf{E}_2 A_{32} + q J_{t \times (s-1)} A_{22} + [(1-q^2) I_{t}+ q^2 J_{t}] A_{32}\\
&= -  q^3 J_{t\times (s-1)} A_{22} -  q^2 J_{t} A_{32} + q J_{t \times (s-1)} A_{22} + [(1-q^2) I_{t}+ q^2 J_{t}] A_{32},
\end{align*}
and hence 
\begin{equation}\label{eqn:A32-1}
A_{32}=-q J_{t\times (s-1)}A_{22}.
\end{equation}
Substituting $A_{32}$ from Eqn.~\eqref{eqn:A32-1} in Eqn.~\eqref{eqn:A22-1}, we get
\begin{align*}
A_{22} &=\dfrac{1}{1-q^2}\left[I_{s-1}- \dfrac{q^2}{q^2(s-1)+1}J_{s-1} -\dfrac{q(1-q^2)}{q^2(s-1)+1} J_{(s-1)\times t} \left(-qJ_{t\times (s-1)}A_{22}\right) \right]\\
&=\dfrac{1}{1-q^2}\left[I_{s-1}- \dfrac{q^2}{q^2(s-1)+1}J_{s-1} +\dfrac{q^2(1-q^2)t}{q^2(s-1)+1} J_{(s-1)} A_{22} \right],
\end{align*}
which implies that $\ds A_{22}\left[ I_{s-1}-\dfrac{q^2t}{q^2(s-1)+1} J_{(s-1)}\right]= \dfrac{1}{1-q^2}\left[I_{s-1}- \dfrac{q^2}{q^2(s-1)+1}J_{s-1}\right].$ From Lemma~\ref{lem:aI+bJ}, $\ds \left[ I_{s-1}-\dfrac{q^2t}{q^2(s-1)+1} J_{(s-1)}\right]^{-1}= I_{s-1}+ \frac{q^2t}{1-q^2(s-1)(t-1)}$, and hence
\begin{align}\label{eqn:A22-12}
A_{22}&= \dfrac{1}{1-q^2}\left[I_{s-1}- \dfrac{q^2}{q^2(s-1)+1}J_{s-1}\right]\left[I_{s-1}+ \frac{q^2t}{1-q^2(s-1)(t-1)}J_{s-1}\right] \nonumber\\
&=\dfrac{1}{1-q^2}\left[I_{s-1}+ \frac{q^2(t-1)}{1-q^2(s-1)(t-1)}J_{s-1}\right].
\end{align}
Thus, using Eqn.~\eqref{eqn:A22-12} in Eqn.~\eqref{eqn:A32-1}, we get
\begin{equation}\label{eqn:A32-12}
A_{32}=-\frac{q}{1-q^2}\left[\frac{1}{1-q^2(s-1)(t-1)}\right]J_{t \times (s-1)}.
\end{equation}
Substituting the expression of $A_{32}$ from Eqn.~\eqref{eqn:A32-1} in Eqn.~\eqref{eqn:A1j},  the expression of $A_{12}$ is given by
$\ds A_{12}  = - q^2 \mathbf{E}_1 A_{22} + q^2  \mathbf{E}_2 J_{t\times (s-1)}A_{22}$. Further, using the identities $\mathbf{E}_2 J_{t\times (s-1)}=t\mathbf{E}_1$ and $\mathbf{E}_1 J_{s-1}=(s-1)\mathbf{E}_1$ from Proposition~\ref{prop:id1}, and the expression of $A_{22}$ from Eqn.\eqref{eqn:A22-12}, we get
\begin{align}\label{eqn:A12-12}
A_{12}&=q^2(t-1)\mathbf{E}_1 A_{22} \nonumber\\
&= q^2(t-1)\mathbf{E}_1 \left[\dfrac{1}{1-q^2}\left[I_{s-1}+ \frac{q^2(t-1)}{1-q^2(s-1)(t-1)}J_{s-1}\right]\right]  \nonumber\\
&=\dfrac{q^2(t-1)}{1-q^2} \left[\mathbf{E}_1 + \frac{q^2(t-1)}{1-q^2(s-1)(t-1)}(s-1)\mathbf{E}_1\right] \nonumber\\
&= \dfrac{q^2}{1-q^2}\left[ \frac{(t-1)}{1-q^2(s-1)(t-1)}\right]\mathbf{E}_1.
\end{align}

We now consider the second row of $\mathscr{F}$ and the third column of $\mathscr{F}^{-1}$ in the product  Eqn.~\eqref{eqn:prod}, we get
$$\mathbf{0}_{(s-1) \times t}=q^2 \mathbf{E}_1^t \widehat{\mathscr{F}} A_{13}+ [(1-q^2) I_{s-1}+q^2 J_{s-1}] A_{23} +  q J_{(s-1)\times t} A_{33}. $$
Substituting $A_{13}$ from Eqn.~\eqref{eqn:A1j}, and using  the identities $\mathbf{E}_1^t \widehat{\mathscr{F}}\mathbf{E}_1=J_{s-1}$ and $ \mathbf{E}_1^t  \widehat{\mathscr{F}} \mathbf{E}_2=J_{(s-1)\times t}$ from Proposition~\ref{prop:id1}, we have
\begin{align*}
\mathbf{0}_{(s-1) \times t}&=q^2 \mathbf{E}_1^t \widehat{\mathscr{F}} [ - q^2  \mathbf{E}_1 A_{23} - q   \mathbf{E}_2 A_{33}]+ [(1-q^2) I_{s-1}+q^2 J_{s-1}] A_{23} +  q J_{(s-1)\times t} A_{33}\\
&=-q^4 \mathbf{E}_1^t \widehat{\mathscr{F}} \mathbf{E}_1 A_{23} - q^3 \mathbf{E}_1^t \widehat{\mathscr{F}}  \mathbf{E}_2 A_{33}+ [(1-q^2) I_{s-1}+q^2 J_{s-1}] A_{23} +  q J_{(s-1)\times t} A_{33}\\
&=-q^4J_{s-1} A_{23} - q^3 J_{(s-1)\times t} A_{33}+ [(1-q^2) I_{s-1}+q^2 J_{s-1}] A_{23} +  q J_{(s-1)\times t} A_{33}, 
\end{align*}
which implies that $[I_{s}+q^2J_{s-1}]A_{23}=-qJ_{s-1}A_{33}.$ From Lemma~\ref{lem:aI+bJ}, we have $\ds [I_{s-1}+q^2J_{s-1}]^{-1}=I_{s-1}-\dfrac{q^2}{q^2(s-1)+1}J_{s-1}$ and hence
\begin{align}\label{eqn:A23-1}
A_{23}&= -q \left[I_{s-1}-\dfrac{q^2}{q^2(s-1)+1}J_{s-1}\right] J_{s-1}A_{33} \nonumber\\
&=- \dfrac{q}{q^2(s-1)+1}J_{s-1}A_{33}.
\end{align}
In Eqn.~\eqref{eqn:prod}, considering the third row of $\mathscr{F}$ and the third column of $\mathscr{F}^{-1}$, we get
$$ I_{t}= q \mathbf{E}_2^t \widehat{\mathscr{F}} A_{13} + q J_{t \times (s-1)} A_{23} + [(1-q^2) I_{t}+ q^2 J_{t}] A_{33}.$$
Substituting $A_{13}$ from Eqn.~\eqref{eqn:A1j}, and using  the identities $\mathbf{E}_2^t \widehat{\mathscr{F}}\mathbf{E}_1=J_{t \times (s-1)}$ and $ \mathbf{E}_2^t  \widehat{\mathscr{F}} \mathbf{E}_2=J_{t}$ from Proposition~\ref{prop:id1}, we have
\begin{align*}
 I_{t}&= q \mathbf{E}_2^t \widehat{\mathscr{F}}[ - q^2  \mathbf{E}_1 A_{23} - q   \mathbf{E}_2 A_{33}] + q J_{t \times (s-1)} A_{23} + [(1-q^2) I_{t}+ q^2 J_{t}] A_{33}\\
 &= - q^3 \mathbf{E}_2^t \widehat{\mathscr{F}}  \mathbf{E}_1 A_{23} - q^2 \mathbf{E}_2^t  \widehat{\mathscr{F}} \mathbf{E}_2 A_{33}  + q J_{t \times (s-1)} A_{23} + [(1-q^2) I_{t}+ q^2 J_{t}] A_{33}\\
 &= - q^3 J_{t \times (s-1)} A_{23} - q^2 J_t A_{33}  + q J_{t \times (s-1)} A_{23} + [(1-q^2) I_{t}+ q^2 J_{t}] A_{33}\\
 &= (1-q^2)[q J_{t \times (s-1)}A_{23} +  A_{33}].
\end{align*}
Now substituting $A_{23}$ from Eqn.\eqref{eqn:A23-1}, the above equation gives 
\begin{align*}
I_{t}&= (1-q^2)\left[q J_{t \times (s-1)}\left[- \dfrac{q}{q^2(s-1)+1}J_{s-1}A_{33}\right] +  A_{33}\right]=  (1-q^2)\left[I_{t}-\dfrac{q^2(s-1)}{q^2(s-1)+1}J_t\right] A_{33}.
\end{align*}
From Lemma~\ref{lem:aI+bJ}, $\ds \left[I_{t}-\dfrac{q^2(s-1)}{q^2(s-1)+1}J_t\right]^{-1}= I_t + \dfrac{q^2(s-1)}{1-q^2(s-1)(t-1)}J_t$, and hence
\begin{equation}\label{eqn:A33-12}
A_{33}= \dfrac{1}{1-q^2}\left[ I_t + \dfrac{q^2(s-1)}{1-q^2(s-1)(t-1)}J_t \right].
\end{equation}
Since $A_{23}=A_{32}^t$, substituting $\ds A_{23}=-\frac{q}{1-q^2}\left[\frac{1}{1-q^2(s-1)(t-1)}\right]J_{(s-1)  \times t}$ using Eqn.~\eqref{eqn:A32-12} and $A_{33}$ from Eqn.~\eqref{eqn:A33-12} in the expression of $A_{13}$ from Eqn.~\eqref{eqn:A1j}, we get
$$A_{13}= - q^2  \mathbf{E}_1 \left[-\frac{q}{1-q^2}\left[\frac{1}{1-q^2(s-1)(t-1)}\right]J_{(s-1)  \times t}\right] - q   \mathbf{E}_2 \left[\dfrac{1}{1-q^2}\left[ I_t + \dfrac{q^2(s-1)}{1-q^2(s-1)(t-1)}J_t \right]\right].$$
Further, using  the identities $\mathbf{E}_1 J_{(s-1)\times t}=(s-1)\mathbf{E}_2$ and $\mathbf{E}_2 J_t= t \mathbf{E}_2$ from Proposition~\ref{prop:id1} gives
\begin{equation}\label{eqn:A13-12}
A_{13}= -\dfrac{q}{1-q^2}\left[\dfrac{1}{1-q^2(s-1)(t-1)}\right]\mathbf{E}_2.
\end{equation}
Since $A_{21}=A_{12}^t$ and $A_{31}=A_{13}^t$,  Eqns.~\eqref{eqn:A12-12} and~\eqref{eqn:A13-12} gives
$$A_{21}= \dfrac{1}{1-q^2}\left[ \frac{q^2(t-1)}{1-q^2(s-1)(t-1)}\right]\mathbf{E}_1^t \mbox{ and } A_{31}= -\dfrac{q}{1-q^2}\left[\dfrac{1}{1-q^2(s-1)(t-1)}\right]\mathbf{E}_2^t.$$ Then, using the identities  $\mathbf{E}_1  \mathbf{E}_1^t= (s-1) \mathbf{E}_{\mathbf{m} \mathbf{m}}$ and  $\mathbf{E}_2  \mathbf{E}_2^t= t \mathbf{E}_{\mathbf{m} \mathbf{m}}$  from Proposition~\ref{prop:id1}, the expression of $A_{11}$ in Eqn.~\eqref{eqn:A1j} is reduces to
\begin{align*}
A_{11} &= \widehat{\mathscr{F}}^{-1} - q^2  \mathbf{E}_1 A_{21} - q   \mathbf{E}_2 A_{31}\\
&=  \widehat{\mathscr{F}}^{-1} - q^2  \mathbf{E}_1 \left[ \dfrac{1}{1-q^2}\left[ \frac{q^2(t-1)}{1-q^2(s-1)(t-1)}\right]\mathbf{E}_1^t\right] - q   \mathbf{E}_2 \left[ -\dfrac{q}{1-q^2}\left[\dfrac{1}{1-q^2(s-1)(t-1)}\right]\mathbf{E}_2^t \right]\\
&=  \widehat{\mathscr{F}}^{-1} -   \dfrac{q^2}{1-q^2}\left[ \frac{q^2(t-1)}{1-q^2(s-1)(t-1)}\right]\mathbf{E}_1 \mathbf{E}_1^t +   \dfrac{q^2}{1-q^2}\left[\dfrac{1}{1-q^2(s-1)(t-1)}\right]\mathbf{E}_2\mathbf{E}_2^t\\
&=  \widehat{\mathscr{F}}^{-1} -   \dfrac{q^2}{1-q^2}\left[ \frac{q^2(s-1)(t-1)}{1-q^2(s-1)(t-1)}\right]\mathbf{E}_{\mathbf{m} \mathbf{m}} +   \dfrac{q^2}{1-q^2}\left[\dfrac{t}{1-q^2(s-1)(t-1)}\right]\mathbf{E}_{\mathbf{m} \mathbf{m}}\\
&=  \widehat{\mathscr{F}}^{-1} +   \dfrac{q^2}{1-q^2}\left[ \frac{(t-1)}{1-q^2(s-1)(t-1)}+1\right]\mathbf{E}_{\mathbf{m} \mathbf{m}}. \\
\end{align*}

In view of induction hypothesis and the above calculations we establish the following:

	\begin{equation*}\label{eqn:F-inv}
		\begin{split}
			&\mathscr{F}^{-1}= \frac{1}{1-q^2} I \\
			&- \frac{q}{1-q^2} \begin{bmatrix}
				\widehat{\mathbf{A}} & \mathbf{0}_{\mathbf{m} \times (s-1)} & \dfrac{1}{1-q^2(s-1)(t-1)} E_2\\
				\mathbf{0}_{(s-1) \times \mathbf{m}} & \mathbf{0}_{s-1} & \dfrac{1}{1-q^2(s-1)(t-1)} J_{(s-1)\times t}\\
				\dfrac{1}{1-q^2(s-1)(t-1)} E_2^t & \dfrac{1}{1-q^2(s-1)(t-1)} J_{t \times (s-1)} & \mathbf{0}_t
			\end{bmatrix}\\
			\\
			&+\frac{q^2}{1-q^2} \begin{bmatrix}
				\widehat{\mathbf{B}} & \dfrac{t-1}{1-q^2(s-1)(t-1)} E_1  & \mathbf{0}_{\mathbf{m} \times t}\\
				\dfrac{t-1}{1-q^2(s-1)(t-1)} E_1^t  & \dfrac{t-1}{1-q^2(s-1)(t-1)} (  J_{s-1}-I_{s-1} ) &  \mathbf{0}_{(s-1)\times t}\\
				\mathbf{0}_{t \times \mathbf{m}} & \mathbf{0}_{t \times (s-1)} & \dfrac{s-1}{1-q^2(s-1)(t-1)} (J_{t} - I_{t})
			\end{bmatrix}\\
			\\
			&+\frac{q^2}{1-q^2} \begin{bmatrix}
				\textup{diag}(\widehat{\mu}) + \left[ \dfrac{t-1}{1-q^2(s-1)(t-1)} + 1 \right] E_{\mathbf{m} \mathbf{m}} & \mathbf{0}_{\mathbf{m} \times (s-1)}  & \mathbf{0}_{\mathbf{m} \times t}\\
				\mathbf{0}_{(s-1) \times \mathbf{m}}  & \dfrac{t-1}{1-q^2(s-1)(t-1)} I_{s-1} &  \mathbf{0}_{(s-1)\times t}\\
				\mathbf{0}_{t \times \mathbf{m}} & \mathbf{0}_{t \times (s-1)} & \dfrac{s-1}{1-q^2(s-1)(t-1)} I_{t}
			\end{bmatrix}\\
			&= \frac{1}{1-q^2} I - \frac{q}{1-q^2} \mathbf{A} +\frac{q^2}{1-q^2} \mathbf{B}  +\frac{q^2}{1-q^2} \textup{diag}(\mu).
\\		\end{split}
	\end{equation*}
This completes the proof.
\end{proof}

\section{Cofactor sum of bi-block graphs}
In this section, we determines the sum of all the cofactors of the exponential distance matrix for bi-block graphs. Notice that the cofactor sum of a matrix is equivalent to the sum of all entries of its adjoint matrix. Denote by $M^*$ the adjoint matrix of an $n \times n$ matrix $M$. Recall that $M^* = \det(M) M^{-1}$ when $M$ is invertible, which means that
\[
\cof(M) = \sum_{i,j} (M^*)_{ij} = \det(M) \sum_{i,j} (M^{-1})_{ij}.
\]

\begin{lem}\label{lem:F-1-cof}
	Let $G$ be the complete bipartite graph $K_{s,t}$ and $\mathscr{F}$ be the exponential distance matrix of $G$. Let $q$ be a nonzero indeterminate such that $q \ne \pm 1$ and $q^2 \ne \dfrac{1}{(s-1)(t-1)}$, then
	\[
	\cof(\mathscr{F}) = (1-q^2)^{s+t-1}  \left(\frac{2q(q-1)st}{(1-q^2)} + s+t \right)
	\]
\end{lem}

\begin{proof}
	From Lemma~\ref{lem:F-1-inv}, we have
	\[
	\mathscr{F}^{-1} = \frac{q}{(q^2-1) \left[q^2(s-1)(t-1)-1\right]}  \begin{bmatrix}
		q(t-1)J_s & -J_{s \times t}\\
		- J_{t \times s} & q(s-1)J_t 
	\end{bmatrix} - \frac{1}{q^2-1} I.
	\]
	Hence
	\begin{equation*}
		\begin{split}
			\sum_{i,j} (\mathscr{F}^{-1})_{ij} &= \frac{q}{(q^2-1) \left[q^2(s-1)(t-1)-1\right]}  \times \left[ q(t-1)s^2 -2st+q(s-1)t^2\right] - \frac{s+t}{q^2-1}\\
			&= \frac{q\times \left[ q(t-1)s^2 -2st+q(s-1)t^2\right]- (s+t)(q^2(s-1)(t-1)-1)}{(q^2-1) \left[q^2(s-1)(t-1)-1\right]}\\
			&= \frac{2q(q-1)st - (s+t)(q^2-1)}{(q^2-1) \left[q^2(s-1)(t-1)-1\right]}.
		\end{split}
	\end{equation*}
	Next, using the fact that $\det(\mathscr{F})= (1-q^2)^{s+t-1} \left[1-q^2(s-1)(t-1)\right]$ from Lemma~\ref{lem:F-1-det} we have the required result.
\end{proof}

\begin{theorem}
	Let $G$ be an \( \mathbf{n} \)-vertex bi-block graph with $r > 1$ blocks $K_{m_i,n_i}$, $i=1,2,\cdots,r$, where $\mathbf{n} = \sum_{i=1}^r (m_i+n_i)-r+1$ and $\mathscr{F}$ be the exponential distance matrix of $G$. Let $q$ be a nonzero indeterminate such that $q \ne \pm 1$ and $q^2 \ne \dfrac{1}{(m_i-1)(n_i-1)}$ for each $1 \le i \le r$. Then
	\begin{align*}
		\cof(\mathscr{F}) &=(1-q^2)^{\mathbf{n}-1} \left[\prod_{i=1}^{r} \left(1-q^2(m_i-1)(n_i-1)\right)\right]\\
		&\quad \quad \times \left[\sum_{i=1}^r \left(\frac{2q(q-1)m_in_i}{(1-q^2)(1-q^2(m_i-1)(n_i-1))} + \frac{m_i+n_i}{1-q^2(m_i-1)(n_i-1)} \right) -(r-1) \right].
	\end{align*}
\end{theorem}

\begin{proof}
	We prove the results using induction on the number of blocks $r$. First assume that $r=1$, then $G$ is an \( \mathbf{n} \)-vertex complete bipartite graph. Without loss of generality we assume $G=K_{s,t}$. Then, using Lemma~\ref{lem:F-1-cof}, we have
	\[
	\cof(\mathscr{F}) = (1-q^2)^{s+t-1}  \left(\frac{2q(q-1)st}{(1-q^2)} + s+t \right),
	\]
	which matches with the expression for $r=1$. Hence the result is true for $r=1$. Now suppose that $r>1$ and the result is true for any bi-block graphs with $r-1$ blocks. Thus, applying inductive hypothesis to $H$, we have
		\begin{align*}
		\cof(\widehat{\mathscr{F}}) &=(1-q^2)^{\mathbf{m}-1} \left[\prod_{i=1}^{r-1} \left(1-q^2(m_i-1)(n_i-1)\right)\right]\\
		&\quad \quad \times \left[\sum_{i=1}^{r-1} \left(\frac{2q(q-1)m_in_i}{(1-q^2)(1-q^2(m_i-1)(n_i-1))} + \frac{m_i+n_i}{1-q^2(m_i-1)(n_i-1)} \right) -(r-2)\right].
	\end{align*}
	Hence 
	\[
	\sum_{i,j} (\hat{\mathscr{F}}^{-1})_{ij} = \sum_{i=1}^{r-1} \left(\frac{2q(q-1)m_in_i}{(1-q^2)(1-q^2(m_i-1)(n_i-1))} + \frac{m_i+n_i}{1-q^2(m_i-1)(n_i-1)} \right) - (r-2).
	\]
	Next, let us consider an \( \mathbf{n} \)-vertex bi-block graph with $r > 1$ blocks, where the leaf block $K_{s,t}$ is attached to $H$ at a cut-vertex. Using the $3 \times 3$ block partition of the matrix $\mathscr{F}^{-1}$ from the proof of the Theorem~\ref{thm:inv}, we have
	\begin{equation*}
		\begin{split}
			\sum_{i,j} (\mathscr{F}^{-1})_{ij} &= \sum_{i,j} (A_{11})_{ij} + \sum_{i,j} (A_{22})_{ij} + \sum_{i,j} (A_{33})_{ij} + 2 \sum_{i,j} (A_{12})_{ij} + 2 \sum_{i,j} (A_{13})_{ij} + 2 \sum_{i,j} (A_{23})_{ij}\\
			&= \sum_{i,j} (\widehat{\mathscr{F}}^{-1})_{ij} + \frac{q^2}{1-q^2} \left[ \frac{t-1}{1-q^2(s-1)(t-1)} +1 \right]\\
			&\quad +  \frac{q^2(t-1)(s-1)^2}{(1-q^2)(1-q^2(s-1)(t-1))} +\frac{s-1}{1-q^2} + \frac{q^2(s-1)t^2}{(1-q^2)(1-q^2(s-1)(t-1))} + \frac{t}{1-q^2}\\
			&\quad + \frac{2q^2}{1-q^2} \left[ \frac{(t-1)(s-1)}{1-q^2(s-1)(t-1)}\right] - \frac{2qt}{(1-q^2)(1-q^2(s-1)(t-1))}\\
			&\quad - \frac{2qt(s-1)}{(1-q^2)(1-q^2(s-1)(t-1))}\\
			&= \sum_{i,j} (\widehat{\mathscr{F}}^{-1})_{ij}\\ & \quad+ \frac{q^2}{(1-q^2)(1-q^2(s-1)(t-1))}			\left[(t-1)+(t-1)(s-1)^2+(s-1)t^2+2(t-1)(s-1)\right] \\
			&\quad + \frac{q^2}{1-q^2} +\frac{s+t-1}{1-q^2}
			- \frac{2qst}{(1-q^2)(1-q^2(s-1)(t-1))}\\
			&= \sum_{i,j} (\widehat{\mathscr{F}}^{-1})_{ij} + \frac{q^2(s^2(t-1)+t^2(s-1)) -2qst}{(1-q^2)(1-q^2(s-1)(t-1))} +\frac{s+t}{1-q^2} -1\\
			&= \sum_{i,j} (\hat{\mathscr{F}}^{-1})_{ij} + \frac{2q(q-1)st + (s+t)(1-q^2)}{(1-q^2) \left[1-q^2(s-1)(t-1)\right]} -1.
		\end{split}
	\end{equation*}
	Since we have taken the leaf block to be $K_{s,t}$, we have $m_r=s$ and $n_r=t$. Thus we have
	\[
	\sum_{i,j} (\mathscr{F}^{-1})_{ij} = \sum_{i,j} (\hat{\mathscr{F}}^{-1})_{ij} + \left(\frac{2q(q-1)m_rn_r}{(1-q^2)(1-q^2(m_r-1)(n_r-1))} + \frac{m_r+n_r}{1-q^2(m_r-1)(n_r-1)} \right) -1.
	\]
	Finally, using the induction hypothesis for $H$, the result follows.
\end{proof}

\section{$q$-Laplacian of bi-block graphs}
Recall that \textbf{deg} and \textbf{adj} represent the diagonal degree matrix and adjacency matrix of a connected graph G, respectively. The Laplacian matrix of $G$ is defined to be $L = \textbf{deg} - \textbf{adj}$. The concept $q$-Laplacian matrix was introduced for trees, when Bapat et al. \cite{Bapat0} investigated the inverse of exponential distance matrix for trees.

For a tree $T$ on $\mathbf{n}$ vertices, let $\mathbf{d}^t = (d_1,d_2,\cdots,d_{\mathbf{n}})$ and $\mathbf{z} = \mathbf{d}-\mathds{1}$, where $d_i$ is the degree of the $i^{th}$ vertex. Then the $q$-Laplacian matrix, $\mathscr{L}$, of a tree $T$ is given by
\[
\mathscr{L} = qL-(q-1)I+q(q-1) \textup{diag}(\mathbf{z}),
\]
where $L$ is the usual Laplacian matrix of $T$. Note that, when $q=1$, then $\mathscr{L} = L$. More interestingly, in \cite{Bapat0} the authors showed the relation between $q$-Laplacian matrix and the inverse of exponential distance matrix for trees. 

\begin{lem}\cite{Bapat0}
	For a tree and any nonzero indeterminate $q \neq \pm 1$,
	\[
	\mathscr{F}^{-1} = \frac{1}{1-q^2} \mathscr{L}.
	\]
\end{lem}

In this section, we extend the definition of $q$-Laplacian matrix to bi-block graphs and show how it is related with the exponential distance matrix.

Let $G$ be a \( \mathbf{n} \)-vertex bi-block graph with $r > 1$ blocks $K_{m_i,n_i}$, $i=1,2,\cdots,r$, where $\mathbf{n} = \sum_{i=1}^r (m_i+n_i)-r+1$. Then, we define an \( \mathbf{n} \)-dimensional vector $\mathbf{x}$ as follows:
\[
\mathbf{x}_G(v) = \sum_{v \in X_i}\frac{1-q^2(m_i-2)(n_i-1)}{1-q^2(m_i-1)(n_i-1)}  + \sum_{v \in Y_i} \frac{1-q^2(m_i-1)(n_i-2)}{1-q^2(m_i-1)(n_i-1)}  - (1-q^2)(k-1),
\]
where $v \in V(G)$ is a vertex with block degree of $v$, $\hat{d}_G(v) = k$ and the summation is taken over all the blocks $K_{m_i,n_i}$ that contains $v$ as a vertex. Note that, when $q=1$ and $m_i=n_i=1$ for all $i=1,2,\cdots,r$, $\mathbf{x}_G(v) = deg(v)$, \textit{i.e.} $\mathbf{x}_G = \mathbf{d}$. Next, we define the $q$-Laplacian matrix for bi-block graphs as follows:
\[
\mathscr{L} = \textup{diag}(\mathbf{x}) + q(q\mathbf{B}-\mathbf{A}),
\]
where the matrices $\mathbf{A}$ and $\mathbf{B}$ are same as defined in where 	$\mathbf{A}, \mathbf{B}$ and $\mu$ as defined in~Eqns.~\eqref{eqn:defn-A} and~\eqref{eqn:defn-B}. One can easily verify that when  $q=1$ and $m_i=n_i=1$ for all $i=1,2,\cdots,r$, $q(qB-A) = \textbf{adj}$, the adjacency matrix of the tree. Hence, $\mathscr{L} = \textbf{deg} - \textbf{adj} = L$, matches with the usual Laplacian matrix for trees. Thus $\mathscr{L}$ is a generalization of the Laplacian matrix for the bi-block graphs.

In the next result, we express the inverse of the exponential distance matrix of the bi-block graphs in terms of the $q$-Laplacian matrix for the same.

\begin{theorem}
	Let $G$ be an \( \mathbf{n} \)-vertex bi-block graph with $r > 1$ blocks $K_{m_i,n_i}$, $i=1,2,\cdots,r$, where $\mathbf{n} = \sum_{i=1}^r (m_i+n_i)-r+1$ and $\mathscr{F}, \mathscr{L}$ be the exponential distance matrix and $q$-Laplacian matrix of $G$, respectively. Let $q$ be a nonzero indeterminate such that $q \ne \pm 1$ and  $\mathscr{F}$ is invertible. Then
	\[
	\mathscr{F}^{-1} = \frac{1}{1-q^2} \mathscr{L}.
	\]
\end{theorem}

\begin{proof}
	From Theorem~\ref{thm:inv}, we have 
	\[
	\mathscr{F}^{-1} = \frac{1}{1-q^2} I - \frac{q}{1-q^2} \mathbf{A} +\frac{q^2}{1-q^2} \mathbf{B}  +\frac{q^2}{1-q^2} \textup{diag}(\mu),
	\]
	\textit{i.e.}
	\[
	(1-q^2)\mathscr{F}^{-1} = I - q \mathbf{A} + q^2 \mathbf{B}  + q^2\textup{diag}(\mu).
	\]
	From the definition of the $q$-Laplacian matrix for bi-block graphs, we have 
	\[
	\mathscr{L} = \textup{diag}(\mathbf{x}) + q(q\mathbf{B}-\mathbf{A}).
	\]
	Thus, if we can show that \[\textup{diag}(\mathbf{x}) = I + q^2\textup{diag}(\mu) \] then we are done. We verify this as follows: Let $v$ be a vertex in $G$ with  $bi_{G}(v)=k$, the block index of $v$. Then,
	\begin{align*}
		1+q^2 \mu(v) &= 1 + q^2\left[\sum_{v \in X_i} \frac{n_i-1}{1-q^2(m_i-1)(n_i-1)} + \sum_{v \in Y_i} \frac{m_i-1}{1-q^2(m_i-1)(n_i-1)} + (k-1)\right]\\
		&= 1+q^2(k-1)+ q^2\left[\sum_{v \in X_i} \frac{n_i-1}{1-q^2(m_i-1)(n_i-1)} + \sum_{v \in Y_i} \frac{m_i-1}{1-q^2(m_i-1)(n_i-1)} \right]\\
		&= (q^2-1)(k-1) + \left[\sum_{v \in X_i} \left(1 + \frac{q^2(n_i-1)}{1-q^2(m_i-1)(n_i-1)} \right) \right.\\ 
		&\qquad  \qquad \qquad \qquad\qquad\qquad\qquad\qquad\qquad \left.  + \sum_{v \in Y_i} \left(1 +  \frac{q^2(m_i-1)}{1-q^2(m_i-1)(n_i-1)} \right) \right]\\		
					&= (q^2-1)(k-1) + \left[\sum_{v \in X_i}\frac{1-q^2(m_i-2)(n_i-1)}{1-q^2(m_i-1)(n_i-1)}  + \sum_{v \in Y_i} \frac{1-q^2(m_i-1)(n_i-2)}{1-q^2(m_i-1)(n_i-1)} \right]\\
		&= \sum_{v \in X_i}\frac{1-q^2(m_i-2)(n_i-1)}{1-q^2(m_i-1)(n_i-1)}  + \sum_{v \in Y_i} \frac{1-q^2(m_i-1)(n_i-2)}{1-q^2(m_i-1)(n_i-1)} - (1-q^2)(k-1)\\
		&=\mathbf{x}(v).
	\end{align*}
	Hence the result follows.
\end{proof}

\small{

}

\end{document}